\documentclass[11pt]{article}
\usepackage{latexsym,amsfonts,amssymb,amsmath,amsthm}
\usepackage{graphicx}

\usepackage[usenames,dvipsnames]{color}
\usepackage{ulem,comment}

\usepackage{color}

\title{On traveling wave solutions in  full parabolic Keller-Segel  chemotaxis systems with logistic source}

\author{
Rachidi B. Salako\\
Department of Mathematics\\
The Ohio State University\\
Columbus OH, 43210-1174\\
\\
 and \\
 \\
 Wenxian Shen\thanks{Partially supported by the NSF grant DMS--1645673} \\
Department of Mathematics and Statistics\\
Auburn University\\
Auburn University, AL 36849 }

\date{}

\begin{document}

\maketitle

\newtheorem{tm}{Theorem}[section]
\newtheorem{prop}{Proposition}[section]
\newtheorem{defin}{Definition}[section] 
\newtheorem{coro}{Corollary}[section]
\newtheorem{lem}{Lemma}[section]
\newtheorem{assumption}{Assumption}[section]
\newtheorem{rk}{Remark}[section]
\newtheorem{nota}[tm]{Notation}
\numberwithin{equation}{section}

\newcommand{\stk}[2]{\stackrel{#1}{#2}}
\newcommand{\dwn}[1]{{\scriptstyle #1}\downarrow}
\newcommand{\upa}[1]{{\scriptstyle #1}\uparrow}
\newcommand{\nea}[1]{{\scriptstyle #1}\nearrow}
\newcommand{\sea}[1]{\searrow {\scriptstyle #1}}
\newcommand{\csti}[3]{(#1+1) (#2)^{1/ (#1+1)} (#1)^{- #1
 / (#1+1)} (#3)^{ #1 / (#1 +1)}}
\newcommand{\RR}[1]{\mathbb{#1}}

\newcommand{\rd}{{\mathbb R^d}}
\newcommand{\ep}{\varepsilon}
\newcommand{\rr}{{\mathbb R}}
\newcommand{\alert}[1]{\fbox{#1}}
\newcommand{\eqd}{\sim}
\def\p{\partial}
\def\R{{\mathbb R}}
\def\N{{\mathbb N}}
\def\Q{{\mathbb Q}}
\def\C{{\mathbb C}}
\def\l{{\langle}}
\def\r{\rangle}
\def\t{\tau}
\def\k{\kappa}
\def\a{\alpha}
\def\la{\lambda}
\def\De{\Delta}
\def\de{\delta}
\def\ga{\gamma}
\def\Ga{\Gamma}
\def\ep{\varepsilon}
\def\eps{\varepsilon}
\def\si{\sigma}
\def\Re {{\rm Re}\,}
\def\Im {{\rm Im}\,}
\def\E{{\mathbb E}}
\def\P{{\mathbb P}}
\def\Z{{\mathbb Z}}
\def\D{{\mathbb D}}
\newcommand{\ceil}[1]{\lceil{#1}\rceil}

\begin{abstract}
This paper is concerned with  traveling wave solutions of the following full parabolic Keller-Segel chemotaxis system with logistic source,
\begin{equation}\label{abs-eq}
\begin{cases}
u_t=\Delta u -\chi\nabla\cdot(u\nabla v)+u(a-bu),\quad x\in\R^N \cr
\tau v_t=\Delta v-\lambda v +\mu u,\quad  x\in \R^N,
\end{cases}
\end{equation}
where $\chi, \mu,\lambda,a,$ and $b$ are positive numbers,  and $\tau\ge 0$.
 Among others, it is proved that if $b>2\chi\mu$ and $\tau \geq \frac{1}{2}(1-\frac{\lambda}{a})_{+} ,$ then for every $c\ge 2\sqrt{a}$, \eqref{abs-eq} has a traveling wave solution $(u,v)(t,x)=(U^{\tau,c}(x\cdot\xi-ct),V^{\tau,c}(x\cdot\xi-ct))$ ($\forall\, \xi\in\R^N$)
connecting the two constant steady states $(0,0)$ and $(\frac{a}{b},\frac{\mu}{\lambda}\frac{a}{b})$,
 and there is no such solutions with speed $c$ less than $2\sqrt{a}$, which  improves considerably the results established in \cite{SaSh3},
and shows that \eqref{abs-eq} has a  minimal wave speed $c_0^*=2\sqrt a$, which is independent of the chemotaxis.

\end{abstract}


\medskip

\hfil\vfil\eject

\section{Introduction}

This work is concerned with traveling wave solutions of the following  full parabolic chemotaxis system

\begin{equation}\label{Main-eq}
\begin{cases}
u_t=\Delta u -\chi\nabla\cdot(u\nabla v)+u(a-bu),\quad x\in\R^N \cr
\tau v_t=\Delta v-\lambda v +\mu u,\quad x\in \R^N,
\end{cases}
\end{equation}
where $\chi,\mu,\lambda,a,$ and $b$ are positive real numbers, $\tau$ is a nonnegative real number,
 and $u(t,x)$ and $v(t,x)$ denote the concentration functions of some mobile species and chemical substance, respectively. { Biologically, the positive constant $\chi$ measures the sensitivity effect on the mobile species by  the chemical substance which is produced overtime by the mobile species;   the reaction $u(a-bu)$ in the first  equation of \eqref{Main-eq} describes the local dynamics of the mobile species; $\lambda$ represents the degradation  rate of the  chemical substance;   and $\mu$ is the rate at which the mobile species produces the chemical substance. The constant $\frac{1}{\tau}$ in the case $\tau>0$  measures the diffusion rate of the chemical substance,
 and the case that $\tau=0$ is  supposed to model
the situation when the chemical substance diffuses very quickly.

System \eqref{Main-eq} is a simplified version of the chemotaxis system proposed by Keller and Segel in their works \cite{KeSe1,KeSe2}.
  Chemotaxis models describe the oriented movements of biological cells and organisms in response to certain chemical substances.
    These mathematical models play very important roles in a wide range of biological phenomena and accordingly a considerable literature is concerned with its mathematical analysis.  The reader is referred to \cite{HiPa,Hor}  for some detailed introduction into the mathematics of Keller-Segel models.

One of the central problems about \eqref{Main-eq} is whether a positive solution blows up at a finite time.   This problem has been studied in many  papers in the case that $a=b=0$ (see \cite{HiPa,Dirk and Winkler,KKAS,kuto_PHYSD,NAGAI_SENBA_YOSHIDA,win_jde,win_JMAA_veryweak,win_arxiv}).   It is shown that finite time blow-up  may occur if either $N=2$ and the total initial population mass is large enough, or $N\geq 3$. It is also shown that some radial solutions to \eqref{Main-eq} in plane collapse into a persistent Dirac-type singularity in the sense that a globally defined measure-valued solution exists which has a singular part beyond some finite time and asymptotically approaches a Dirac measure (see \cite{LSV1,TeWi1}). We refer the reader to \cite{BBTW,HoSt} and the references therein for more insights in the studies of chemotaxis models.

When the constant $a$ and $b$ are positive, the finite time blow-up phenomena in \eqref{Main-eq} may be suppressed  to some  extent. In fact in this case, it is known that when the space dimension is equal to one or two, solutions to \eqref{Main-eq} on bounded domains with Neumann boundary conditions and initial functions in a space of certain integrable functions are defined for all time. And it is enough for the self limitation coefficient $b$ to be big enough comparing  to the chemotaxis sensitivity coefficient to prevent finite time blow-up, see \cite{ITBWS16,SaSh1,TeWi2}.

 Traveling wave solutions constitute  another class of important solutions of \eqref{Main-eq}. Observe that, when $\chi=0$, the chemotaxis system \eqref{Main-eq} reduces to
\begin{equation}\label{fisher-kpp}
u_t=\Delta u+u(a-bu), \quad x\in\R^N.
\end{equation}
Due to the pioneering works of Fisher
\cite{Fisher} and Kolmogorov, Petrowsky, Piskunov \cite{KPP} on traveling wave solutions and take-over properties of \eqref{fisher-kpp},
\eqref{fisher-kpp} is also referred to as  the Fisher-KPP equation.
 The following results are well known about traveling wave solutions of \eqref{fisher-kpp}.
Equation \eqref{fisher-kpp} has  traveling wave solutions of the form  $u(t,x)=\phi(x\cdot\xi-ct)$ ($\xi\in S^{N-1}$)
connecting $0$ and $\frac{a}{b}$ $(\phi(-\infty)=\frac{a}{b},\phi(\infty)=0)$ of all speeds $c\geq 2\sqrt a$ and
has no such traveling wave solutions of slower
speed.    $c^*_0=2\sqrt a$ is therefore the minimal wave speed of traveling wave solutions of \eqref{fisher-kpp} connecting $0$ and $\frac{a}{b}$.   Since the pioneering works by  Fisher \cite{Fisher} and Kolmogorov, Petrowsky,
Piscunov \cite{KPP},  a huge amount research has been carried out toward the front propagation dynamics of
  reaction diffusion equations of the form,
\begin{equation}
\label{general-fisher-eq}
u_t=\Delta u+u f(t,x,u),\quad x\in\R^N,
\end{equation}
where $f(t,x,u)<0$ for $u\gg 1$,  $\partial_u f(t,x,u)<0$ for $u\ge 0$ (see \cite{ArWe2,BHN,BeHaNa1,BeHaNa2,Henri1,Fre,FrGa,LiZh,LiZh1,Nad,NoRuXi,NoXi1,She1,She2,Wei1,Wei2,Zla}, etc.).
}
\medskip

In \cite{SaSh3}, the authors of the current paper studied the existence of traveling wave solutions of \eqref{Main-eq} connecting
the two constant steady states $(0,0)$ and $(\frac{a}{b},\frac{\mu}{\lambda}\frac{a}{b})$. Roughly, it is proved in \cite{SaSh3} that,
when the chemotaxis sensitivity $\chi$ is small relative to the logistic damping $b$, \eqref{Main-eq} has traveling wave solutions connecting
$(0,0)$ and $(\frac{a}{b},\frac{\mu}{\lambda}\frac{a}{b})$ with speed $c$, which is bounded below by some constant $c^*>c_0^*=2\sqrt a$ and is bounded 
above by some constant $c^{**}<\infty$. But many fundamental questions remain open, for example, whether \eqref{Main-eq} has traveling wave solutions connecting
$(0,0)$ and $(\frac{a}{b},\frac{\mu}{\lambda}\frac{a}{b})$ with speed $c\gg 1$; whether there is a minimal wave speed of  traveling wave solutions
of \eqref{Main-eq} connecting
$(0,0)$ and $(\frac{a}{b},\frac{\mu}{\lambda}\frac{a}{b})$, and if so, how the chemotaxis affects the minimal wave speed.

The objective of the current paper is to investigate those fundamental open questions. To state the main results of the current paper,
we first introduce the definition of traveling wave solutions of \eqref{Main-eq} and the induced problems to be studied.

\subsection{Traveling wave solutions and the induced problems}

{\it An entire solution} of \eqref{Main-eq} is a classical solution  $(u(t,x),v(t,x))$ of \eqref{Main-eq} which is defined for all $x\in\R^N$ and $t\in\R$. Note that the constant solutions $(u(t,x),v(t,x))=(0,0)$ and  $(u(t,x),v(t,x))=(\frac{a}{b},\frac{\mu a}{\lambda b})$ are clearly two particular entire solutions of \eqref{Main-eq}.  An entire solution of \eqref{Main-eq} of the form $(u(t,x),v(t,x))=(U^{\tau,c}(x\cdot\xi-ct),V^{\tau,c}(x\cdot\xi-ct))$ for some unit vector $\xi\in S^{N-1}$ and
 some constant $c\in\R$ is called a {\it traveling wave solution} with speed $c$.
 A traveling wave solution $(u(t,x),v(t,x))=(U^{\tau,c}(x\cdot\xi-ct),V^{\tau,c}(x\cdot\xi -ct))$ $(\xi\in\ S^{N-1}$) of \eqref{Main-eq} with speed c is said to connect $(0,0)$ and $(\frac{a}{b},\frac{\mu a}{\lambda b})$  if
\begin{equation}\label{Main-TW-eq}
\liminf_{x\to-\infty}U^{\tau,c}(x)=\frac{a}{b} \quad \text{and}\quad \limsup_{x\to\infty}U^{\tau,c}(x)=0.
\end{equation}
 We say that a traveling wave solution $(u(t,x),v(t,x))=(U^{\tau,c}(x\cdot\xi -ct),V^{\tau,c}(x\cdot\xi -ct))$ of \eqref{Main-eq} is nontrivial and connects $(0,0)$ at one end if
\begin{equation}\label{Persisence-TW-eq}
\liminf_{x\to-\infty}U^{\tau,c}(x)>0 \quad \text{and}\quad \limsup_{x\to\infty}U^{\tau,c}(x)=0.
\end{equation}

\smallskip

 Observe that for given $c\in\R$, a traveling solution $(u(t,x),v(t,x))=(U^{\tau,c}(x\cdot\xi-ct),V^{\tau,c}(x\cdot\xi -ct))$ $(\xi\in\ S^{N-1}$)  of \eqref{Main-eq} with speed $c$ connecting the states $(0,0)$ and $(\frac{a}{b},\frac{\mu a}{\lambda b})$ gives rise to a stationary solution $(u,v)=(U^{\tau,c}(x),V^{\tau,c}(x))$  of the parabolic-elliptic system
\begin{equation}\label{Rewrite-main-eq1}
\begin{cases}
u_t=u_{xx}+((c-\chi v_{x})u)_x +(a-bu)u, \quad x\in\R\cr
0=v_{xx}+\tau cv_x-\lambda v+\mu u,\quad x\in\R.
\end{cases}
\end{equation}
connecting the states $(0,0)$ and $(\frac{a}{b},\frac{\mu a}{\lambda b})$.

Conversely, if $(u,v)=(U^{\tau,c}(x),V^{\tau,c}(x))$ is a stationary solution of \eqref{Rewrite-main-eq1} connecting the states $(0,0)$ and $(\frac{a}{b},\frac{\mu a}{\lambda b})$, then $(u(t,x),v(t,x))=(U^{\tau,c}(x\cdot\xi-ct),V^{\tau,c}(x\cdot\xi -ct))$  is a traveling wave solution of \eqref{Main-eq} with speed $c$ connecting the states $(0,0)$ and $(\frac{a}{b},\frac{\mu a}{\lambda b})$ for any  $\xi\in\ S^{N-1}$.

To study traveling wave solutions of \eqref{Main-eq} with speed $c$ connecting the states $(0,0)$ and $(\frac{a}{b},\frac{\mu a}{\lambda b})$
is then equivalent to study stationary solutions of  \eqref{Rewrite-main-eq1} connecting the states $(0,0)$ and $(\frac{a}{b},\frac{\mu a}{\lambda b})$.
It is clear that \eqref{Rewrite-main-eq1} is equivalent to
\begin{equation}\label{Rewrite-main-eq2}
\begin{cases}
u_t=u_{xx}+(c-\chi v_{x})u_x +(a-\chi v_{xx}-bu)u,\quad x\in\R\cr
0=v_{xx}+\tau cv_x-\lambda v+\mu u,\quad x\in\R.
\end{cases}
\end{equation}
Hence, to study  traveling wave solutions of \eqref{Main-eq} connecting the states $(0,0)$ and $(\frac{a}{b},\frac{\mu a}{\lambda b})$ we shall study steady states solutions of \eqref{Rewrite-main-eq2} connecting the states $(0,0)$ and $(\frac{a}{b},\frac{\mu a}{\lambda b})$.

 Before stating  the main results of the current paper, we  next recall some existing results on the existence of solutions of
\eqref{Rewrite-main-eq2}  with given initial functions and existence of  steady states solutions of \eqref{Rewrite-main-eq2}
 or traveling wave solutions of \eqref{Main-eq} connecting the states $(0,0)$ and $(\frac{a}{b},\frac{\mu a}{\lambda b})$.

 \subsection{Existing results}

 Let
$$
C_{\rm unif}^b(\R)=\{u\in C(\R)\,|\, u(x)\quad \text{is uniformly continuous in}\,\,\, x\in\R\quad \text{and}\,\, \sup_{x\in\R}|u(x)|<\infty\}
$$
equipped with the norm $\|u\|_\infty=\sup_{x\in\R}|u(x)|$.

\begin{prop}[Local existence]
\label{local-existence-prop}
For every nonnegative initial function $u_0\in C^b_{\rm unif}(\R)$ and $c\in\R$ , there is a unique maximal time $T_{max}(u_0)$, such that \eqref{Rewrite-main-eq2} has a unique classical  solution $(u(t,x;u_0,c)$, $v(t,x;u_0,c))$  defined for every $x\in\R$ and $0\le t<T_{\max}(u_0)$ with
$u(0,x;u_0,c)=u_0(x)$. Moreover if $T_{max}(u_0)<\infty$ then
$$
\lim_{t\to T_{\max}(u_0)-}\|u(t,\cdot;u_0,c)\|_{\infty}=\infty.
 $$
\end{prop}

The above proposition can proved by similar arguments as those in  \cite[Theorem 1.1]{SaSh1}.
The following proposition follows from the arguments of \cite[Theorems A and B]{SaSh3} (it is proved in \cite[Theorems A and B]{SaSh3}
for the case that $\lambda=\mu=1$).

\begin{prop}[Global existence]
\label{global-existence-prop}
Consider \eqref{Rewrite-main-eq2}.
\begin{itemize}
\item[(1)]  Assume that $0\le \frac{\chi \mu\tau c}{2\sqrt \lambda}{ < }{b-\chi\mu}$. Then
 for any  $u_0\in C_{\rm unif}^b(\R)$ with $0\le u_0$,  $T_{\max}(u_0)=\infty$.  Moreover,
 \begin{equation*}\label{global-exixst-thm-eq1}
 \|u(t,\cdot;u_0,c)\|_{\infty}\leq \max\{\|u_0\|_{\infty},\frac{a}{b-\chi\mu -\frac{\chi \mu\tau c}{2\sqrt \lambda}}\}\end{equation*}
 for every $t\geq 0$.

\item[(2)]  Assume that { $0\leq \frac{\chi\mu\tau c}{\sqrt \lambda}<b-2\chi\mu$}. Then for any $u_0\in C_{\rm unif}^b(\R)$ with $\inf_{x\in\R}u_0(x)>0$,
$$
\lim_{t\to\infty}\Big[\|u(t,\cdot;u_0,c)-\frac{a}{b}\|_\infty+\|v(\cdot,t;u_0,c)-\frac{\mu}{\lambda}\frac{a}{b}\|_\infty\Big]=0.
$$
\end{itemize}
\end{prop}

\begin{prop}
\label{traveling-wave-prop}
\begin{itemize}
\item[(1)] For every $\tau >0$, there is $0<\chi_{\tau}^*<\frac{b}{2\mu}$ such that for every $0<\chi<\chi_{\tau}^*$, there exist  two positive numbers $0< c^{*}(\chi,\tau)<c^{**}(\chi,\tau)$ satisfying that for every $ c\in   ( c^{*}(\chi,\tau)\ ,\ c^{**}(\chi,\tau))$, \eqref{Main-eq}  has a traveling wave solution $(u,v)=(U(x\cdot\xi-ct),V(x\cdot\xi-ct))$ $(\forall\,\xi\in S^{N-1})$  connecting the constant solutions $(0,0)$ and $(\frac{a}{b},\frac{\mu}{\lambda}\frac{a}{b})$.  Moreover,
$$
\lim_{\chi\to 0+}c^{**}(\chi,\tau)=\infty,$$
$$
\lim_{\chi\to 0+}c^{*}(\chi,\tau)=\begin{cases}
2\sqrt{a}\qquad \qquad  \qquad\ \text{if} \quad 0<a\leq  \frac{\lambda+\tau a}{(1-\tau)_+}\cr
\frac{\lambda+\tau a}{(1-\tau)_{+}}+\frac{a(1-\tau)_{+}}{\lambda+\tau a}\quad \text{if} \quad  a\geq \frac{\lambda+\tau a}{(1-\tau)_+},
\end{cases}
$$
and
$$
\lim_{x\to \infty}\frac{U(x;\tau)}{e^{- k x}}=1,
$$
where $k$ is the only solution of the equation $k+\frac{a}{k}=c$ in the interval $(0\ ,   \min\{\sqrt{a}, \sqrt{\frac{\lambda+\tau a}{(1-\tau)_+}}\})$.


\item[(2)] For any given $\tau\ge 0$ and $\chi\ge 0$,   \eqref{Main-eq} has no traveling wave  solutions $(u,v)=(U(x\cdot\xi-ct),V(x\cdot\xi-ct))$
$(\forall\,\, x\in S^{N-1})$
with $(U(-\infty),V(-\infty))=(\frac{a}{b},\frac{\mu}{\lambda}\frac{a}{b})$, $(U(\infty),V(\infty))=(0,0)$, and $c<2 \sqrt a$.
\end{itemize}
\end{prop}

 As mentioned before, in the absence of chemotaxis (i.e. $\chi=0$), $c_0^*=2\sqrt a$ is the minimal wave speed of
 the Fisher-KPP equation \eqref{fisher-kpp}.
Both biologically and mathematically, it is interesting to know whether the results stated in Proposition \ref{traveling-wave-prop}(1) can be improved to the following: {\it for any $c\ge c_0^*$, \eqref{Main-eq} has a traveling wave solution
$(u(t,x),v(t,x))=(U(x\cdot\xi-ct),V(x\cdot\xi-ct))$
 $(\forall\, \xi\in S^{N-1})$ connecting $(\frac{a}{b},\frac{\mu}{\lambda}\frac{a}{b})$ and $(0,0)$}, which implies  that \eqref{Main-eq}
has a minimal wave speed, and the chemotaxis does not affect the magnitude of the minimal wave speed.

Also, as mentioned before,
the objective of the current paper is to investigate the above open problems or to improve  the results obtained in \cite{SaSh3}.  Roughly, we will show that there is no upper bound for the speeds of traveling wave solutions of \eqref{Main-eq} and under some natural conditions, $c_0^*=2\sqrt a$ is the minimal wave speed of
\eqref{Main-eq}. The precise statements of the main results are stated in next subsection.

\subsection{The statements of the main results}

In order to state our main results, we first introduce some notations. For given $c\in\R$, let
$$B_{\lambda,c,\tau}=\frac{1}{\sqrt{4\lambda+\tau^2c^2}},$$
 $$\lambda_1^c=\frac{(\tau c+\sqrt{4\lambda+\tau^2c^2})}{2},\quad \lambda_2^c=\frac{(\sqrt{4\lambda+\tau^2c^2}-\tau c)}{2},$$
  and
$$
c_{\kappa}=\frac{a+\kappa^2}{\kappa}\quad  \forall\,\,  0<\kappa<\sqrt{a}.
$$
Note that $\lambda_2^c$ and $-\lambda_1^c$ are the positive and negative roots of the quadratic equations
$$
m^2+\tau c m -\lambda =0.
$$
Note also that
\begin{equation}\label{product-sum-id}
\lambda_1^c\lambda_2^c=\lambda,\quad
\lambda_1^c+\lambda_2^c=\frac{1}{B_{\lambda,c,\tau}}.
\end{equation}
 All the above quantities are defined for any $\tau\ge 0$.

Throughout this work, we shall always suppose that $c>0$. This restriction is justified by the fact that \eqref{Main-eq} does not have a non-trivial traveling wave with speed $c\leq 0$ (see Proposition \ref{traveling-wave-prop}(2)).

Note that, by \eqref{product-sum-id},
\begin{equation}\label{Intr-eq1}
\frac{\lambda_{2}^cB_{\lambda,c,\tau}}{\lambda_2^c+\kappa}\Big(\kappa -\frac{\lambda}{\lambda_1^c}\Big)_+=\frac{\lambda_2^c(\kappa-\lambda_2^c)_+}{(\lambda_2^c+\lambda_1^c)(\kappa+\lambda_2^c)}<1.
\end{equation}
Hence the following quantity is well defined
\begin{equation}
\label{b-star-eq}
b_\tau^*=\sup\{ 1+ \frac{\lambda_2^{c_\kappa}(\kappa-\lambda_2^{c_{\kappa}})_+}{(\lambda_2^{c_\kappa}+\lambda_1^{c_\kappa})(\kappa+\lambda_2^{c_\kappa})} \, | \, 0<\kappa<\sqrt{a}\}.
\end{equation}
 It is clear that $b_\tau^*$ is defined for all $\tau\ge 0$,   $b_{\tau}^*\leq 2$ for all $\tau\ge 0$, and
$b_0^*=1+\frac{(\sqrt a-\sqrt \lambda)_+}{2(\sqrt a+\sqrt \lambda)}$.

\smallskip

For the sake of simplicity in the statements of our results, let us introduce the following  standing hypotheses.

\smallskip
\smallskip

\noindent {\bf (H1)}  $b>\chi\mu$.

\smallskip

\noindent {\bf (H2)}   $ b > b_\tau^* \chi\mu$.

\smallskip

\noindent {\bf(H3)}  $b> 2\chi\mu$.

\smallskip

\noindent {\bf (H4)} $\tau\geq \frac{1}{2}\left(1-\frac{\lambda}{a}\right)_+$.

\smallskip
\smallskip

Observe that {\bf (H3)} implies {\bf (H2)}, and {\bf (H2)} implies {\bf (H1)}.

\smallskip
\smallskip

The following results about the global existence of bounded classical solutions and the stability of the positive constant equilibria of \eqref{Rewrite-main-eq2} will be of great use in our arguments.

\begin{tm} \label{Global-existence and Stability}  For any $\tau\ge 0$ and $c>0$, the following hold.
\begin{description}
\item[(i)]  If ${\bf (H1)}$
 holds, then for every $u_0\in C^b_{\rm unif}(\R)$, with $u_0\ge0$, \eqref{Rewrite-main-eq2} has a unique global classical solution $(u(t,x;u_0,c),v(t,x;u_0,c))$ on $(0,\infty)\times\R$ satisfying $\lim_{t\to 0^+}\|u(0,\cdot;u_0,c)-u_0(\cdot)\|_{\infty}=0$. Moreover it holds that
\begin{equation}\label{uniform-upper-bound}
\|u(t,\cdot;u_0,c)_{\infty}\|\le \max\Big\{\|u_0\|_{\infty}, \frac{a}{b-\chi\mu}\Big\},\quad t\geq 0.
\end{equation}
\item[(ii)] If ${\bf (H3)}$
holds, then for every $u_0\in C^b_{\rm unif}(\R)$, with $\inf_{x\in\R}u_0(x)>0$, we have that
\begin{equation}\label{stability-eq}
\lim_{t\to\infty}\Big(\|u(t,\cdot;u_0,c)-\frac{a}{b}\|_{\infty}+\|v(t,\cdot;u_0,c)-\frac{a\mu}{b\lambda}\|_{\infty}\Big)=0.
\end{equation}
\end{description}

\end{tm}

\begin{rk} When $\tau=0$, we recover \cite[Theorems 1.5 $\&$ 1.8]{SaSh1}. For $\tau>0$, Theorem \ref{Global-existence and Stability} improves the results stated in Proposition \ref{global-existence-prop}.
\end{rk}

\medskip

 Observe  that the function
$$
(0,\sqrt{a})\ni\kappa\mapsto \lambda_1^{c_{\kappa}}-\kappa
$$
is strictly decreasing. Hence the quantity
\begin{equation}
\label{kappa-star-eq}
\kappa_\tau^*:=\sup\{0<\kappa<\sqrt{a}\,|\, \lambda_1^{c_\kappa}-\kappa\ge 0\}
\end{equation}
is well defined. It holds that
$$
\lambda_1^{c_{\kappa}}-\kappa>0
$$
whenever $0<\kappa<\kappa_\tau^*$. Note also that 
\begin{equation}
\label{kappa-star-eq1}
\kappa^*_{\tau}=\min\left\{\sqrt{a},\sqrt{\frac{\lambda +\tau a}{(1-\tau)_+}}\right\}.
\end{equation}
Indeed, it holds that  $\lambda_1^{c_{\sqrt{a}}}>\sqrt{a}$ for every $\tau\ge 1$. On the other hand, for $0\leq \tau <1$, if $ \lambda_1^{c_{\kappa}}=\kappa$ for some $0<\kappa\leq \sqrt{a}$, then  it holds that
$$
 \lambda +\kappa \tau c_{\kappa}-\kappa^2 =0 \quad \Leftrightarrow \quad \lambda +\tau a= (1-\tau)\kappa^2 \quad \kappa=\sqrt{\frac{\lambda+\tau a}{1-\tau}}.
$$
Hence \eqref{kappa-star-eq1} holds.

Let
\begin{equation}
\label{c-star-eq}
c^*(\tau)=\kappa_\tau^*+\frac{a}{\kappa_\tau^*}.
\end{equation}
 Note that $\kappa_\tau^*$ and $c^*(\tau)$ are defined for all $\tau\ge 0$, and
 $$\kappa_0^*=\min\{\sqrt \lambda,\sqrt a\},\quad c^*(0)=\min\{\sqrt \lambda,\sqrt a\}+\frac{a}{\min\{\sqrt\lambda,\sqrt a\}}.
 $$

 We have the following theorem on the existence of traveling wave solutions of \eqref{Main-eq}.

\begin{tm}\label{existence-of-TW}
 For any $\tau\ge 0$, the following hold.
\begin{itemize}
\item[(1)]
 If {\bf (H2)} holds, then for any $c>c^*(\tau)$,  \eqref{Main-eq} has a nontrivial traveling wave solution $(u,v)(t,x)=(U(x\cdot\xi-c_{\kappa}t),V(x\cdot\xi-c_{\kappa}t))$ $(\forall\,\,\xi\in S^{N-1})$  satisfying \eqref{Persisence-TW-eq}, where  $\kappa\in (0,\kappa_\tau^*)$ is such that $c_\kappa=c$. Furthermore, it holds that
\begin{equation}\label{asyp-beh-at-positive-infinty}
\lim_{x\to\infty}\frac{U(x)}{e^{-\kappa x}}=1.
\end{equation}
If in addition, ${\bf (H3)}$ holds, then
\begin{equation}\label{asyp-beh-at-negative-infinty}
\lim_{x\to-\infty}|U(x)-\frac{a}{b}|=0.
\end{equation}

\item[(2)]  If {\bf (H2)} and {\bf (H4)} hold, then $\kappa_{\tau}^*=\sqrt a$ and $c^*(\tau)=2\sqrt a$. Hence for any $c>2\sqrt a$, the results in (1) hold true.

\item[(3)] Suppose that {\bf (H3)} holds. Then \eqref{Main-eq} has a traveling wave solution $(u,v)(t,x)=(U^{\tau,c}(x\cdot\xi-ct,V^{\tau,c}(x\cdot\xi-ct))$ $(\forall\, \xi\in S^{N-1})$  with speed $c^*(\tau)$ connecting $(0,0)$ and $(\frac{a}{b},\frac{a\mu}{b\lambda})$.
\end{itemize}
\end{tm}

\begin{rk}\label{Remark on TW-1}
{
\begin{itemize}

\item[(1)] Note that  the conditions in Proposition \ref{traveling-wave-prop} are $\chi<\chi_\tau^*$ and $b>2\chi\mu$, which imply
both {\bf (H2)} and {\bf (H3)}. Hence the assumptions in Theorem \ref{existence-of-TW}(1) are weaker than those in Proposition \ref{traveling-wave-prop} for the existence of traveling wave solutions.  Note also that, by Theorem   \ref{existence-of-TW}(1),
the lower bound $c^*(\tau)$ for the wave speed is independent of $\chi$, and the upper bound is $\infty$. By the proof of \cite[Theorem C]{SaSh3}, $\kappa^*_\tau=\min\{\sqrt{a},\frac{\lambda +\tau a}{(1-\tau)_+}\}$ is an upper bound found for the decay rate of traveling wave solutions found in \cite{SaSh3}. Hence  $c^*(\chi,\tau)\ge c_{\kappa^*_\tau}=c^*(\tau)$, that is,  the lower bound provided in Theorem \ref{existence-of-TW}
for the wave speed of traveling wave solutions of \eqref{Main-eq} is not larger than that provided in  Proposition \ref{traveling-wave-prop}.
Moreover, under the assumptions (H2) and (H4), $c^*(\tau)=2\sqrt a<c^*(\chi,\tau)$. 
Therefore Theorem \ref{existence-of-TW}  improves considerably Proposition \ref{traveling-wave-prop}.

\item[(2)]  Recall that $b^*_0=1+\frac{(\sqrt{a}-\sqrt{\lambda})_+}{2(a+\sqrt{\lambda})}$, $\kappa^*_0=\min\{\sqrt{a},\sqrt{\lambda}\}$, and $c^*(0)=\min\{\sqrt{a},\sqrt{\lambda}\}+\frac{a}{\min\{\sqrt{a},\sqrt{\lambda}\}}$. Hence Theorem \ref{existence-of-TW} in the case $\tau=0$ recovers \cite[Theorem 1.4]{SaShXu}.

\item[(3)]   When $\lambda\ge a$, $c^*(\tau)=c_0^*=2\sqrt a$ for any ${ \tau\ge 0}$.  Hence if $\lambda\geq a$ and $0<\chi\mu<\frac{b}{2}$ hold, by Theorem \ref{existence-of-TW} for every ${\tau\ge 0}$ and $c\geq 2\sqrt{a}$, \eqref{Main-eq} has a  traveling wave solution $(u,v)(t,x)=(U^{\tau,c},V^{\tau,c})(x-ct)$ with speed  $c$ connecting $(0,0)$ and $(\frac{a}{b},\frac{a\mu}{b\lambda})$. Whence, if $\lambda\geq a$ and $0<\chi<\frac{b}{2\mu} $,  Theorem \ref{existence-of-TW}  implies that $c_0^*=2\sqrt a$ is the minimal wave speed
of traveling wave solutions of \eqref{Main-eq} connecting $(0,0)$ and  $(\frac{a}{b},\frac{a\mu}{b\lambda})$, and that
the chemotaxis does not affect the magnitude of the minimal wave speed of \eqref{Main-eq}.  Biologically, $\lambda\ge a$ means that  the degradation  rate $\lambda$ of the  chemical substance is greater than the intrinsic growth rate $a$ of the mobile species, and $0<\chi\mu<\frac{b}{2}$
indicates that the product of the chemotaxis sensitivity $\chi$ and the rate $\mu$ at which the mobile species produces the chemical substance is less than half of the logistic damping $b$.

\item[(4)] When
$\lambda<a$, $c^*(\tau)=c_0^*=2\sqrt a$ for $\tau> \frac{1}{2}\left(1-\frac{\lambda}{a}\right)$. Hence if $\lambda<a$ and $0<\chi\mu<\frac{b}{2}$ hold, by Theorem \ref{existence-of-TW} for every $\tau>\frac{1}{2}(1-\frac{\lambda}{a})$ and $c\ge 2\sqrt{a}$, \eqref{Main-eq} has a  traveling wave solution $(u,v)(t,x)=(U^{\tau,c},V^{\tau,c})(x-ct)$ with speed  $c$ connecting $(0,0)$ and $(\frac{a}{b},\frac{a\mu}{b\lambda})$. Thus in this case, Theorem \ref{existence-of-TW} also  implies that $c_0^*=2\sqrt a$ is the minimal wave speed
of traveling wave solutions of \eqref{Main-eq} connecting $(0,0)$ and  $(\frac{a}{b},\frac{a\mu}{b\lambda})$, and that
the chemotaxis does not affect the magnitude of the minimal wave speed of \eqref{Main-eq}.  Biologically,
$\tau> \frac{1}{2}\left(1-\frac{\lambda}{a}\right)$ indicates that diffusion rate of the chemical substance is not big.

\item[(5)] By Theorem \ref{existence-of-TW} it holds that $c^*(\tau)=2\sqrt{a}$ whenever $\tau\geq \frac{1}{2}$ and
\eqref{Main-eq} has a minimal wave speed, which is $c^*(\tau)$. When $\lambda<a$ and $ 0\le \tau< \frac{1}{2}$, it remains open whether \eqref{Main-eq}
 has a minimal wave speed, and if so, whether the minimal wave speed equals $2\sqrt{a}$.  It would be interesting to study the stability of the traveling wave solutions of \eqref{Main-eq}. When $\tau=0$, the spreading speeds  of solutions of \eqref{Main-eq} with compactly supported initial functions are studied in \cite{SaShXu}. It would be also interesting to study these spreading results when $\tau>0,$ which we plan to carry out in our future work.
\end{itemize}

}
\end{rk}

The rest of the paper is organized as follow. In Section 2, we prove some preliminaries results to use in the subsequent sections. Section 3 is devoted to the proof of Theorem \ref{Global-existence and Stability}, while the proof of Theorem \ref{existence-of-TW} will be presented in Section 4.

\section{Preliminary lemmas}

In this section, we prove some lemmas to be used in the proofs of the main results in the later sections.
 Throughout of this section, we assume $\tau\ge 0$.

{For every $u\in C^b_{\rm unif}(\R)$ and $c\in\R$, let
\begin{equation}\label{psi-definition}
\Psi(x;u,c,\tau)=\mu\int_{0}^{\infty}\int_{\R}\frac{e^{-\lambda s}e^{-\frac{|x+\tau cs-y|^2}{4s}}}{\sqrt{4\pi s}}u(y)dyds.
\end{equation}
It is well known that $\Psi(x;u,c,\tau)\in C^2_{\rm unif}(\R)$ and solves the elliptic equation
$$
\frac{d^2}{dx^2}\Psi(x;u,c,\tau)+\tau c\frac{d}{dx}\Psi(x;u,c,\tau)-\lambda\Psi(x;u,c,\tau)+\mu u=0.
$$

\begin{lem}
\label{new-lm1} It holds that
\begin{align}\label{psi-definition-eq2}
\Psi(x;u,c,\tau)
=&\frac{\mu}{\sqrt{4\lambda+\tau^2c^2}}\int_{\R}e^{\frac{-\sqrt{4\lambda+\tau^2c^2}|x-y|-\tau c(x-y)}{2}}u(y)dy\cr
=&\mu B_{\chi,c,\tau}\Big(e^{-\lambda_1^cx}\int_{-\infty}^{x}e^{\lambda_1y}u(y)dy+e^{\lambda_2 x} \int_{x}^{\infty}e^{-\lambda_2^cy}u(y)dy\Big)
\end{align}
 and
\begin{align}\label{space-derivative-of-psi-1}
\frac{d}{dx}\Psi(x;u,c,\tau)=&\mu B_{\chi,c,\tau}\Big(-\lambda_1^ce^{-\lambda_1^cx}\int_{-\infty}^{x}e^{\lambda_1^cy}u(y)dy
+\lambda_2^ce^{\lambda_2^cx}\int_{x}^{\infty}e^{-\lambda_2^cy}u(y)dy\Big).
\end{align}
\end{lem}

\begin{proof}
 For the case that $\tau=0$, the lemma is proved in \cite[Lemma 2.1]{SaShXu}.

 In the following, we prove the case that $\tau>0$.  Observe that it is enough to prove the result for $\tau=1$. The general case follows by replacing $c$ by $\tau c$. So, without loss of generality, we set $\tau=1$. First, observe that the following identity holds,
\begin{equation}\label{nnew-eq1}
\int_{0}^{\infty}\frac{e^{-\frac{\beta^2}{ 4s}-s }}{\sqrt{4 \pi s}}ds=\frac{e^{-\beta}}{2}, \quad \forall \beta >0.
\end{equation}

Next using Fubini's Theorem, one can exchange the order of integration  in \eqref{psi-definition} to obtain
\begin{align}\label{nnew-eq2}
\Psi(x;u,c,1)=&\mu \int_{0}^{\infty}\int_{\R}\frac{e^{-\lambda s}e^{-\frac{|x+cs-y|^2}{4s}}}{\left[4\pi s\right]^{\frac{1}{2}}}u(y)dyds\cr
=&\mu \int_{\R}\Big[\int_{0}^{\infty}\frac{e^{-\frac{|x+cs-y|^2}{4s}-\lambda s}}{\sqrt{4\pi s}}ds\Big]u(y)dy\cr
=&\int_{\R}e^{-\frac{c(x-y)}{2}}\Big[\int_0^\infty \frac{e^{-\big[ \frac{(x-y)^2}{4s} +\frac{(4\lambda+ c^2)}{4}s \big]}}{\sqrt{4\pi s}} ds\Big]u(y)dy
\end{align}
By the change of variable $z=\frac{(4\lambda +c^2)s}{4}$ and taking $\beta =\frac{\sqrt{4\lambda+c^2}}{2}|x-y|$, it follows from \eqref{nnew-eq1} that
$$
\int_0^\infty \frac{e^{-\big[ \frac{(x-y)^2}{4s} +\frac{(4\lambda +c^2)}{4}s \big]}}{\sqrt{4\pi s}} ds=\frac{2}{\sqrt{4\lambda+c^2}}\int_{0}^{\infty}\frac{e^{-\frac{\beta^2}{ 4z}-z }}{\sqrt{4 \pi z}}dz =\frac{1 }{\sqrt{4\lambda+c^2}}e^{-\frac{\sqrt{4\lambda+c^2}|x-y|}{2}}.
$$
This together with   \eqref{nnew-eq2} implies that
$$
\Psi(x;u,c,1)=\frac{\mu}{\sqrt{4\lambda+c^2}}\int_{\R}e^{\frac{-\sqrt{4\lambda+c^2}|x-y|-c(x-y)}{2}}u(y)dy.
$$
Thus \eqref{psi-definition-eq2} holds. Note that \eqref{space-derivative-of-psi-1} then follows from a direction calculation.
\end{proof}

\begin{lem}\label{lem-001}
For every $u\in C^b_{\rm unif}(\R)$, $u(x)\ge 0$, it holds that
\begin{equation}\label{estimate-on-space-derivative-1}
|\frac{d}{dx}\Psi(x;u,c,\tau)|\le \lambda_1^c\Psi(x;u,c,\tau),\ \quad  \forall\ x\in\R, \ c\in\R.
\end{equation}
 Furthermore,  it holds that
 \begin{equation}\label{super-solution-ineq}
 \chi\kappa\Psi_x(\cdot;u,c,\tau)-\chi\Psi_{xx}(\cdot;u,c,\tau)\le \chi\mu\Big(\frac{ B_{\lambda,c,\tau}\left((\tau c+\kappa)\lambda_2 -\lambda \right)_+}{(\lambda_2+\kappa)} +1\Big)Me^{-\kappa x}
 \end{equation}
whenever $0\leq u(x)\leq Me^{-\kappa x}$ for some  $\kappa\ge 0$ and $M>0$.

In particular, if \begin{equation}\label{Main-Hypothesis-1-c}
 \chi\mu\Big(\frac{B_{\lambda,c,\tau}\big((\tau c+\kappa)\lambda_2 -\lambda \big)_+}{(\lambda_2+\kappa)} +1\Big) \leq b,
\end{equation}
 holds, then
\begin{equation}\label{estimate-on-space-derivative-2}
\chi\kappa \Psi_{x}(x;u,c,\tau)-\chi \Psi_{xx}(x;u,c,\tau)-be^{-\kappa x}\le 0, \quad \forall x\in\R,
\end{equation}
whenever $0\leq u(x)\leq e^{-\kappa x}$ for some positive real numbers $\kappa>0$ and $M>0$.
\end{lem}

 \begin{proof}
  For the case that $\tau=0$, the lemma is proved in \cite[Lemma 2.2]{SaShXu}.  In the following, we prove the lemma for any $\tau\ge 0$.

 First, by \eqref{psi-definition-eq2} and \eqref{space-derivative-of-psi-1}, we have
 \begin{align*}
|\frac{d}{dx}\Psi(x;u,c,\tau)|& \le \frac{ \sqrt{4\lambda+\tau^2c^2}+\tau c}{2}\Psi(x;u,c,\tau).
\end{align*}
 This implies \eqref{estimate-on-space-derivative-1}.

Next, we prove \eqref{estimate-on-space-derivative-2}. It follows from \eqref{psi-definition} and \eqref{space-derivative-of-psi-1} that
\begin{align}\label{A-e1}
\chi\kappa \Psi_{x}(x;u,c,\tau)-\chi \Psi_{xx}(x;u,c,\tau)
=&\chi\kappa \Psi_{x}(x;u,c,\tau)-\chi (\lambda \Psi(x;u,c,\tau)-\tau c\Psi_{x}(x;u,c,\tau)-\mu u)\cr
=& \chi(\tau c+\kappa)\Psi_{x}(x;u,c,\tau)-\chi\lambda\Psi(x;u,c,\tau)+\chi\mu u\cr
=&- \chi\mu B_{\lambda,c,\tau} \left((\tau c+\kappa)\lambda_1^c+\lambda\right)e^{-\lambda_1^cx}\int_{-\infty}^{x}e^{\lambda_1^cy}u(y)dy\cr
&+\chi\mu B_{\lambda,c,\tau}\left((\tau c+\kappa)\lambda_2^c -\lambda \right)e^{\lambda_2^cx}\int_{x}^{\infty}e^{-\lambda_2y}u(y)dy+\chi\mu u.
\end{align}
Hence, since $0\leq u\leq Me^{-\kappa x}$, it follows that
\begin{align*}
\chi\left(\kappa \Psi_{x}(x;u,c,\tau)-\Psi_{xx}(x;u,c,\tau)\right)\leq &  \chi\mu B_{\lambda,c,\tau}\left((\tau c+\kappa)\lambda_2^c -\lambda \right)_+Me^{\lambda_2^c x}\int_{x}^{\infty}e^{-\lambda_2^c y}e^{-\kappa y}dy+\frac{\chi\mu M}{ e^{\kappa x}}\cr
=&\chi\mu\Big(\frac{ B_{\lambda,c,\tau}\left((\tau c+\kappa)\lambda_2^c -\lambda \right)_+}{(\lambda_2^c+\kappa)} +1 \Big)Me^{-\kappa x}\cr
\end{align*}
 Hence, \eqref{super-solution-ineq} follows.
\end{proof}

\begin{rk}\label{Remark 1} Observe that
\begin{align}
\label{aux-eq1}
\tau c \lambda_2^c-\lambda=& \frac{\tau c}{2}\left(\sqrt{4\lambda +\tau^2c^2}-\tau c\right)-\lambda\cr
=& \frac{2\lambda \tau c}{\sqrt{4\lambda +\tau^2c^2}+\tau c}-\lambda\cr
=&-\frac{\lambda \lambda_2^c}{\lambda_1^c}<0.
\end{align}
Hence
$$\frac{B_{\lambda,c,\tau}}{\lambda_2^c}(\tau c\lambda_2^c-\lambda)_+=0,$$
and
$$
\frac{B_{\lambda,c,\tau}}{\lambda_2^c+\kappa}\Big( (\tau c+\kappa)\lambda_2^c-\lambda \Big)_+=\frac{\lambda_{2}^cB_{\lambda,c,\tau}}{\lambda_2^c+\kappa}\Big(\kappa -\frac{\lambda}{\lambda_1^c}\Big)_+.
$$
We also note from \eqref{product-sum-id}  that
\begin{equation}
\label{aux-eq2}
 B_{\lambda,c,\tau}\Big(\frac{\lambda}{\lambda_1^c}+\frac{\lambda}{\lambda_2^c}\Big)=1.
 \end{equation}
These identities will be frequently used later.
\end{rk}

For every $0<\kappa<\tilde{\kappa}<\sqrt{a}$ with $\tilde{\kappa}<2\kappa$ and $M,D\ge 1$, consider the functions $\varphi_{\kappa}(x)$, $\overline{U}_{\kappa,D}(x)$, and $\underline{U}_{\kappa,D} (x)$  given by  $$
\varphi_{\kappa}(x)=e^{-\kappa x},
$$
\begin{equation}\label{U-minus-def}
U^-_D(x)=\varphi_{\kappa}(x)-D\varphi_{\tilde{\kappa}}(x), \quad x\in\R,
\end{equation}
\begin{equation}\label{super-sol-def}
\overline{U}_{\kappa,M}(x)=\min\{M,\varphi_\kappa(x)\},
\end{equation}
and
\begin{equation}\label{sub-sol-def}
\underline{U}_{\kappa,D}(x)=\begin{cases}
\varphi_{\kappa}(x)-D\varphi_{\tilde{\kappa}}(x), \quad x\geq \overline{x}_{\kappa,D}\cr
\varphi_{\kappa}(x_{\kappa,D})-D\varphi_{\tilde{\kappa}}(x_{\kappa,D}), \quad x\leq \overline{x}_{\kappa,D},
\end{cases}
\end{equation}
where $\overline{x}_{\kappa,D}$ satisfies
\begin{equation}\label{l-005}
\max\{\varphi_{\kappa}(x)-D\varphi_{\tilde{\kappa}}(x)\, |\, x\in\R\}=\varphi_{\kappa}(\overline{x}_{\kappa,D})-D\varphi_{\tilde{\kappa}}(\overline{x}_{\kappa,D}).
\end{equation}
Letting $ \underline{x}_{\kappa,D}:=\frac{\ln(D)}{\tilde{\kappa}-\kappa}$, there holds that
$$
U^-_D(x)\begin{cases}>0 \quad \,\text{if}\,\, x>\underline{x}_{\kappa,D},\cr
<0, \quad \text{if}\,\, x<\underline{x}_{\kappa,D}.
\end{cases}
$$

For every $u\in C^b_{\rm unif}(\R)$, let
\begin{equation}\label{E1}
\mathcal{A}_{u,c}(U)=U_{xx}+(c-\chi\Psi_x(\cdot;u,c,\tau))U_x+(a-\chi\Psi_{xx}(\cdot;u,c,\tau)-bU)U.
\end{equation}
\begin{lem}\label{lem-002}  For given $\tau\ge 0$, assume that  {\bf (H2)} holds and $\kappa<\kappa_\tau^*$.
Then there is $D^*>1$ such that for every $D\geq D^*,$ $M>0$, and  $$u\in\tilde{\mathcal{E}}:=\{u\in C^{b}_{\rm unif}(\R)\, |\, \max\{U^-_D(x),0 \}\leq u(x)\leq \min\{M, \varphi_{\kappa}(x)\} \ \forall\ x\in\R \}
$$ it holds that
\begin{equation}\label{sub-sol-eq}
\mathcal{A}_{u,c_{\kappa}}(U^{-}_D)\ge 0 \quad \forall x\in (\underline{x}_{\kappa,D},\infty).
\end{equation}

\end{lem}

\begin{proof}
 We first note that {\bf (H2)} implies \eqref{Main-Hypothesis-1-c}, and $\kappa<\kappa_\tau^*$ implies
\begin{equation}\label{Main-Hypothesis-2}
\lambda_1^{c_\kappa}>\kappa.
\end{equation}

Let $u\in\tilde{\mathcal{E}}$ be given and  $U^-(x)=U^-_D(x)$.   Then
\begin{align*}
\mathcal{A}_{u,c_\kappa}(U^{-})=& U^{-}_{xx}+(c_\kappa-\chi  \Psi_{x}(\cdot;u,c_\kappa))U^{-}_{x}+(a-\chi\Psi_{xx}-bU^{-})U^{-}\cr
=& \left( \kappa^2 e^{-\kappa x}-\tilde{\kappa}^2De^{-\tilde{\kappa}x}\right)+(c_\kappa-\chi\Psi_x)(-\kappa e^{-\kappa x}+\tilde{\kappa}De^{-\tilde{\kappa}x})+a(e^{-\kappa x}-De^{-\tilde{\kappa}x})\cr
& -(\chi\Psi_{xx}+bU^{-})U^-\cr
=&D(\tilde{\kappa}c_\kappa-\tilde{\kappa}^2-a)e^{-\tilde{k}x}-\chi\Psi_x(-\kappa e^{-\kappa x}+\tilde{\kappa}De^{-\tilde{\kappa}x})-(\chi(\lambda\Psi -\mu u -\tau c_\kappa\Psi_x)+bU^{-})U^-\cr
=&D A_{\kappa}e^{-\tilde{k}x}-\chi\Psi_x(-\kappa e^{-\kappa x}+\tilde{\kappa}De^{-\tilde{\kappa}x})-(\chi\lambda\Psi -\chi\mu u -\tau c_\kappa\chi\Psi_x+bU^{-})U^-\cr
\ge& D A_{\kappa}e^{-\tilde{k}x}+\chi\underbrace{\Psi_x(\kappa e^{-\kappa x}-\tilde{\kappa}De^{-\tilde{\kappa}x})}_{\mathbb{I}_1}+\underbrace{(-\chi\lambda\Psi  +\tau c_\kappa\chi\Psi_x-(b-\chi\mu)U^{-})U^-}_{\mathbb{I}_2}.
\end{align*}
where $A_{\kappa}:=\tilde{\kappa}c_\kappa-\tilde{\kappa}^2-a$. Next, observe that since $\lambda_1^{c_\kappa}>\kappa$, it holds that
\begin{align*}
\mathbb{I}_1=&\mu B_{\lambda,c_\kappa,\tau}\Big(-\lambda_1^{c_\kappa}e^{-\lambda_1^{c_\kappa}x}\int_{-\infty}^{x}e^{\lambda_1^{c_\kappa}y}u(y)dy +\lambda_2^{c_\kappa}e^{\lambda_2^{c_\kappa}x}\int_x^{\infty}e^{-\lambda_2^{c_\kappa}x}u(y)\Big)(\kappa e^{-\kappa x}-\tilde{\kappa}De^{-\tilde{\kappa}x})\cr
\ge & -\mu B_{\lambda,c_\kappa,\tau}\Big(\kappa \lambda_1^{c_\kappa}e^{-(\lambda_1^{c_\kappa}+\kappa)x}\int_{-\infty}^{x}e^{\lambda_1^{c_\kappa}y}u(y)dy +\tilde{\kappa}D\lambda_2^{c_\kappa}e^{(\lambda_2^{c_\kappa}-\tilde{\kappa})x}\int_x^{\infty}e^{-\lambda_2^{c_\kappa}y}u(y)\Big)\cr
\geq &  -\mu B_{\lambda,c_\kappa,\tau}\Big(\kappa \lambda_1^{c_\kappa}e^{-(\lambda_1^{c_\kappa}+\kappa)x}\int_{-\infty}^{x}e^{\lambda_1^{c_\kappa}y}e^{-\kappa y}dy +\tilde{\kappa}D\lambda_2^{c_\kappa}e^{(\lambda_2^{c_\kappa}-\tilde{\kappa})x}\int_x^{\infty}e^{-\lambda_2^{c_\kappa}y}e^{-\kappa y}\Big)\cr
=&-\mu B_{\lambda,c_\kappa,\tau}\Big(\frac{\kappa\lambda_1^{c_\kappa}}{\lambda_1^{c_\kappa}-\kappa}e^{-(2\kappa-\tilde{\kappa}) x}+\frac{\tilde{\kappa}D\lambda_2^{c_\kappa}}{\lambda_2^{c_\kappa}+\kappa}e^{-\kappa x}\Big)e^{-\tilde{\kappa}x}
\end{align*}
and
\begin{align*}
\mathbb{I}_2=&\chi\mu B_{\lambda,c_\kappa,\tau}\Big(-(\tau c_\kappa+\lambda)\lambda_1^{c_\kappa} e^{-\lambda_1^{c_\kappa} x}\int_{-\infty}^xe^{\lambda_1^{c_\kappa} y}u(y)+(\tau c_\kappa -\lambda)\lambda_2^{c_\kappa}e^{\lambda_2^{c_\kappa}x}\int_x^{\infty}e^{-\lambda_2^{c_\kappa} y}u(y)dy\Big)U^-\cr
&-(b-\chi\mu)(e^{-2\kappa x} -DU^-(x)e^{-\tilde{\kappa}x}-De^{-(\tilde{\kappa}+\kappa)x})\cr
\geq &-
\chi \mu B_{\lambda,c_\kappa,\tau}\Big(\frac{(\tau c_\kappa+\lambda)\lambda_1^{c_\kappa}}{ e^{(\lambda_1^{c_\kappa}+\kappa)x}}\int_{-\infty}^xe^{\lambda_1^{c_\kappa} y}u(y)dy+(\tau c_\kappa -\lambda)_{-}\lambda_2^{c_\kappa}U^-(x)e^{\lambda_2^{c_\kappa}x}\int_x^{\infty}e^{-\lambda_2^{c_\kappa} y}u(y)dy\Big)\cr
&-(b-\chi\mu)(e^{-2\kappa x} -De^{-(\tilde{\kappa}+\kappa)x})\cr
\geq & -
\chi\mu B_{\lambda,c_\kappa,\tau}\Big(\frac{(\tau c_\kappa+\lambda)\lambda_1^{c_\kappa}}{ e^{(\lambda_1^{c_\kappa}+\kappa)x}}\int_{-\infty}^xe^{\lambda_1^{c_\kappa} y}u(y)dy+(\tau c_\kappa -\lambda)_{-}\lambda_2^ce^{(\lambda_2^c-\kappa)x}\int_x^{\infty}e^{-\lambda_2^c y}u(y)dy\Big)\cr
&-(b-\chi\mu)(e^{-2\kappa x} -De^{-(\tilde{\kappa}+\kappa)x})\cr
\geq & -
\chi\mu B_{\lambda,c_\kappa,\tau}\Big(\frac{(\tau c_\kappa+\lambda)\lambda_1^{c_\kappa}}{ e^{(\lambda_1^c+\kappa)x}}\int_{-\infty}^xe^{\lambda_1^c y}e^{-\kappa y}dy+(\tau c_\kappa -\lambda)_{-}\lambda_2^{c_\kappa}e^{(\lambda_2^{c_\kappa}-\kappa)x}\int_x^{\infty}e^{-\lambda_2^{c_\kappa} y}e^{-\kappa y}dy\Big)\cr
&-(b-\chi\mu)(e^{-2\kappa x} -De^{-(\tilde{\kappa}+\kappa)x})\cr
=& -
\chi\mu B_{\lambda,c_\kappa,\tau}\Big(\frac{(\tau c_\kappa+\lambda)\lambda_1^{c_\kappa}}{\lambda_1^{c_\kappa}-\kappa}+\frac{(\tau c_\kappa -\lambda)_{-}\lambda_2^{c_\kappa}}{\lambda_2^{c_\kappa}+\kappa}\Big)e^{-2\kappa x}-(b-\chi\mu)(e^{-2\kappa x} -De^{-(\tilde{\kappa}+\kappa)x}).
\end{align*}
Thus, with $D>1$, $0<\kappa_1:=2\kappa-\tilde{\kappa}<\kappa$, and $x>\underline{x}_{\kappa,D}>0$, it holds that
\begin{align*}
\frac{\mathcal{A}(U^-)}{e^{-\tilde{\kappa}x}}\ge & \Big( DA_{\kappa} -\Big[ \chi\mu B_{\lambda,c,\tau}\big(\frac{(\kappa+(\tau c_\kappa+\lambda))\lambda_1^{c_\kappa}}{\lambda_1^c-\kappa}+\frac{(\tilde{\kappa}D+(\tau c_\kappa-\lambda)_+)\lambda_2^{c_\kappa} }{\lambda_2^{c_\kappa}+\kappa}\big) +(b-\chi\mu)\Big]e^{-\kappa_1 \underline{x}_{\kappa,D}}\Big).
\end{align*}
Setting $\tilde{\kappa}=\kappa+\eta$, we have $A_{\kappa}>0$,
$$
e^{-\kappa_1\underline{x}_{\kappa,D}}=e^{-\frac{(\kappa-\eta)}{\eta}\ln(D)}=\frac{1}{D^{\frac{\kappa-\eta}{\eta
}}}.
$$
Therefore, for $0<\eta<\min\{\frac{\kappa}{2},\sqrt{a}-\kappa\}$, it holds that $$\kappa<\tilde{\kappa}=\kappa+\eta<\min\{2\kappa,\sqrt{a}\},$$
$$
\frac{\kappa-\eta}{\eta}>1,
$$
 and
$$
\lim_{D\to\infty}\Big( DA_{\kappa} -\Big[ \chi\mu B_{\lambda,c,\tau}\big(\frac{(\kappa+D(c+\lambda))\lambda_1^c}{\lambda_1^c-\kappa}+\frac{(\tilde{\kappa}D+(c-\lambda)_+)\lambda_2^c }{\lambda_2^c+\kappa}\big) +(b-\chi\mu)\Big]e^{-\kappa_1 \underline{x}_{\kappa,D}}\Big)=\infty.
$$
Therefore, there is $D^*>1$ such that \eqref{sub-sol-eq} holds for every $D\geq D^*$ and $u\in\tilde{\mathcal{E}}$.
\end{proof}

\section{Proof of Theorem \ref{Global-existence and Stability}}

In this section, we prove Theorem \ref{Global-existence and Stability}.

\begin{proof}[Proof of  Theorem \ref{Global-existence and Stability}]

(1) Let $(u(t,x;u_0,c),v(t,x;u_0,c))$ be defined on $[0,T_{\max})$. Note by Proposition \ref{local-existence-prop} that in order to show that $T_{\max}=\infty$, it is enough the prove that \eqref{uniform-upper-bound} holds.  For every $T\in(0,T_{\max})$ let $M_T:=\sup_{0\leq t\le T}\|u(t,\cdot;u_0,c)\|_{\infty}$. With $\kappa=0$ and $M=M_T$, it follows from \eqref{super-solution-ineq}  that
$$
u_t\leq u_{xx} +(c-\chi v_x)u_x+\Big(a +\chi\mu \big( \frac{B_{\lambda,c,\tau}(\tau c\lambda_2^c-\lambda)_+}{\lambda_2^c}+1\big)M_T-bu \Big)u, \quad 0<t<T
$$
Hence, by comparison principle for parabolic equations, it holds that
$$
\|u(t,\cdot;u_0,c)\|_{\infty}\leq \max\Big\{\|u_0\|_{\infty}, \frac{a +\chi\mu \Big( \frac{B_{\lambda,c,\tau}(\tau c\lambda_2^c-\lambda)_+}{\lambda_2^c}+1\Big)M_T}{b}\Big\}, \quad \forall\ t\in[0,T].
$$
Hence, if $M_T> \|u_0\|_{\infty}$, we must have
$$
M_T\leq \frac{ a +\chi\mu\Big(\frac{B_{\lambda,c,\tau}(\tau c\lambda_2^c-\lambda)_{+}}{\lambda_2^c}+1 \Big) M_T }{b}.
$$
 By \eqref{aux-eq1},  $(\tau c\lambda_2^c-\lambda)_{+}=0$. Hence
$$
M_T\leq  \frac{a}{b-\chi\mu}.
$$
Thus, it holds that
$$
M_T\leq \max\Big\{\|u_0\|_{\infty},\frac{a}{b-\chi\mu}\Big\}, \quad\ 0< T<T_{\max}.
$$
Which yield that $T_{\max}=\infty$, and by Remark \ref{Remark 1} we conclude that \eqref{uniform-upper-bound} holds.

\smallskip

(2) We show that \eqref{stability-eq} holds. We follow the ideas of the proof of \cite[Theorem  1.8]{SaSh1}.

Let
$$\overline{u}=\limsup_{t\to\infty}\|u(t,\cdot;u_0,c)\|_{\infty}\quad {\rm and}\quad \underline{u}:=\liminf_{t\to\infty}\inf_{x\in\R}u(t,x;u_0,c).
$$
Since $\inf_{x\in\R}u_0(x)>0$, it follows from the arguments of \cite[Theorem 1.2 (i) ]{SaSh6_I} that
 $$0<\underline{u}\leq \overline{u}<\infty.
  $$
  It suffices to prove that
  \begin{equation}
  \label{new-aux-eq1}
  \underline{u}=\bar u=\frac{a}{b}.
  \end{equation}

  To this end, for every $T>0$, let
  $$\overline{u}_T:=\sup_{t\ge T}\sup_{x\in\R}u(t,x;u_0,c)\quad {\rm and}
  \quad \underline{u}_T:=\inf_{t\ge T}\inf_{x\in\R}u(t,x;u_0,c).
   $$
   Let
   $$\mathcal{L}(u)=u_{xx}+(c-\chi v_x)u_x.
   $$
  By  \eqref{A-e1} (with $\kappa=0$),
    for every $t\ge T$ and $x\in\R$, there holds
\begin{align*}
u_t\leq & \mathcal{L}(u) +\Big(a-\chi\mu B_{\lambda,c,\tau}(\tau c\lambda_1^c+\lambda)e^{-\lambda_1^cx}\int_{-\infty}^{x}e^{\lambda_1^c y}\underline{u}_Tdy -(b-\chi\mu)u\Big)u\cr
&+\chi\mu B_{\lambda,c,\tau}\Big( (\tau c\lambda_2^c-\lambda)_+e^{\lambda_2^c x}\int_{x}^\infty e^{-\lambda_2^c y}\overline{u}_Tdy -(\tau c\lambda_2^c-\lambda)_{-}e^{\lambda_2^c x}\int_{x}^\infty e^{-\lambda_2^c y}\underline{u}_Tdy \Big)u\cr
=& \mathcal{L}(u)+\Big(a+\chi\mu B_{\lambda,c,\tau}\big(-(\tau c+\frac{\lambda}{\lambda_1^c})\underline{u}_T+(\tau c-\frac{\lambda}{\lambda_2^c})_{+}\overline{u}_T-(\tau c-\frac{\lambda}{\lambda_2^c})_{-}\underline{u}_T\big) -(b-\chi\mu)u\Big)u.
\end{align*}
Hence, by comparison principle for parabolic equations, it holds that
$$
(b-\chi\mu)\overline{u}\leq a+\chi\mu B_{\lambda,c,\tau}\Big(-(\tau c+\frac{\lambda}{\lambda_1^c})\underline{u}_T+(\tau c-\frac{\lambda}{\lambda_2^c})_{+}\overline{u}_T-(\tau c-\frac{\lambda}{\lambda_2^c})_{-}\underline{u}_T\Big).
$$
Letting $T\to\infty$, we obtain
\begin{equation}\label{A-eq2}
(b-\chi\mu)\overline{u}\leq a+ \chi\mu B_{\lambda,c,\tau}\Big( -\frac{(\tau c\lambda_1^c+\lambda)}{\lambda_1^c}\underline{u} +\frac{(\tau c\lambda_2^c-\lambda)_+}{\lambda_2^c}\overline{u}-\frac{(\tau c\lambda_2^c-\lambda)_-}{\lambda_2^c}\underline{u}\Big).
\end{equation}
Similarly, from \eqref{A-e1} (with $\kappa=0$) it follows for every $t\ge T$ and $x\in\R$ that
\begin{align*}
u_t\geq & \mathcal{L}(u) +\Big(a-\chi\mu B_{\lambda,c,\tau}(\tau c\lambda_1^c+\lambda)e^{-\lambda_1^cx}\int_{-\infty}^{x}e^{\lambda_1^c y}\overline{u}_Tdy -(b-\chi\mu)u\Big)u\cr
&+\chi\mu B_{\lambda,c,\tau}\Big( (\tau c\lambda_2^c-\lambda)_+e^{\lambda_2^c x}\int_{x}^\infty e^{-\lambda_2^c y}\underline{u}_Tdy -(\tau c\lambda_2^c-\lambda)_{-}e^{\lambda_2^c x}\int_{x}^\infty e^{-\lambda_2^c y}\overline{u}_Tdy \Big)u\cr
=& \mathcal{L}(u) +\Big(a+\chi\mu B_{\lambda,c,\tau}\big(-(\tau c+\frac{\lambda}{\lambda_1^c})\overline{u}_T+(\tau c-\frac{\lambda}{\lambda_2^c})_{+}\underline{u}_T-(\tau c-\frac{\lambda}{\lambda_2^c})_{-}\overline{u}_T\big) -(b-\chi\mu)u\Big)u.
\end{align*}
Hence, by comparison principle for parabolic equations, it holds that
$$
(b-\chi\mu)\overline{u}\geq a+\chi\mu B_{\lambda,c,\tau}\Big(-(\tau c+\frac{\lambda}{\lambda_1^c})\overline{u}_T+(\tau c-\frac{\lambda}{\lambda_2^c})_{+}\underline{u}_T-(\tau c-\frac{\lambda}{\lambda_2^c})_{-}\overline{u}_T\Big).
$$
Letting $T\to\infty$, we obtain that
\begin{equation}\label{A-eq3}
(b-\chi\mu)\underline{u}\geq a+ \chi\mu B_{\lambda,c,\tau}\Big( -\frac{(\tau c\lambda_1^c+\lambda)}{\lambda_1^c}\overline{u} +\frac{(\tau c\lambda_2^c-\lambda)_+}{\lambda_2^c}\underline{u}-\frac{(\tau c\lambda_2^c-\lambda)_-}{\lambda_2^c}\overline{u}\Big).
\end{equation}
Since $(\tau c\lambda_2^c-\lambda)_+=0$ by \eqref{aux-eq1}, by adding side-by-side inequalities \eqref{A-e1} and \eqref{A-eq2}, we obtain
\begin{align*}
(b-\chi\mu)(\overline{u}-\underline{u})\leq & \chi\mu B_{\lambda,c,\tau}\Big(\frac{ \tau c\lambda_1^c+\lambda}{\lambda_1^c}+\frac{(\lambda-\tau c\lambda_2^c)}{\lambda_2^c} \Big)(\overline{u}-\underline{u})\cr
=&\chi\mu B_{\lambda,c,\tau}\Big( \frac{\lambda}{\lambda_1^c}+\frac{\lambda}{\lambda_2^c}\Big)(\overline{u}-\underline{u}).
\end{align*}
 By \eqref{aux-eq2}, $ B_{\lambda,c,\tau}\Big( \frac{\lambda}{\lambda_1^c}+\frac{\lambda}{\lambda_2^c}\Big)=1$.
Thus, since {\bf (H3)} holds, we conclude that
$\underline{u}=\overline{u}$. By \eqref{aux-eq1}, \eqref{A-eq2}, and \eqref{A-eq3},
\begin{align*}
(b-\chi\mu)\underline{u}&= a+ \chi\mu B_{\lambda,c,\tau}\Big( -\frac{(\tau c\lambda_1^c+\lambda)}{\lambda_1^c}\underline {u}
+\frac{\tau c\lambda_2^c-\lambda}{\lambda_2^c}\underline{u}\Big)\\
&=a-\chi\mu \underline{u}.
\end{align*}
This implies \eqref{new-aux-eq1} and (2) thus follows.
\end{proof}

\section{Proof of Theorem \ref{existence-of-TW}}

In this section, following the techniques developed in \cite{SaSh3}, we present the proof of Theorem \ref{existence-of-TW}.
Without loss of generality, we assume that $N=1$ in \eqref{Main-eq}.

Through this section we suppose that ${\bf (H2)}$ holds and $0<\kappa<\kappa_\tau^*$. We choose $0<\eta<\min\{2\kappa,\sqrt{a}-\kappa\}$ and set $\tilde{\kappa}=\kappa+\eta$ and $M=\frac{a}{b-\chi\mu}$. We fix a constant $D\geq D^*$, where $D^*$ is given by Lemma \ref{sub-sol-def}. Define
$$
\mathcal{E}:=\{u\in C^b_{\rm}(\R)\ :\ \underline{U}_{\kappa,D}\leq u\leq \overline{U}_{\kappa,M}\}
$$ where $\overline{U}_{\kappa,M}$ and $\underline{U}_{\kappa,D}$ are given by \eqref{super-sol-def} and \eqref{sub-sol-eq} respectively. For every $u\in\mathcal{E}$, we let $U(t,x;u)$ denote the solution of the parabolic equation
\begin{equation}
\begin{cases}
U_t=\mathcal{A}_{u,c_{\kappa}}(U), \quad\ x\in\R, t>0\cr
U(0,x)=\overline{U}_{\kappa,M},\quad x\in\R.
\end{cases}
\end{equation}
The following result holds.

\begin{lem}\label{lem-003}
\begin{itemize}
\item[(i)] For every $u\in\tilde{ \mathcal{E}}$, the function $U(t,x)\equiv M$ satisfies $\mathcal{A}_{u,c_\kappa}(U)\leq 0$ on $\R\times\R$.
\item[(ii)] For every $u\in\tilde{ \mathcal{E}}$, the function $U(t,x)=e^{-\kappa x}$ satisfies $\mathcal{A}_{u,c_\kappa}(U)\leq 0$ on $\R\times\R$.
\item[(iii)] For every $u\in\tilde{ \mathcal{E}}$, the function $U(t,x)=U^-_D$, where $U^-_D$ is given by \eqref{U-minus-def}, satisfies $\mathcal{A}_{u,c_\kappa}(U)\geq 0$ on $\R\times(\underline{x}_{\kappa,D},\infty)$.
\item[(iv)] Suppose that {\bf (H3)} holds. There $0<\delta\ll 1$ such that for every $u\in\tilde{ \mathcal{E}}$, the function $U(t,x)=\delta$ satisfies $\mathcal{A}_{u,c_\kappa}(U)\geq 0$ on $\R\times\R$.
\end{itemize}
\end{lem}
\begin{proof}
Using Lemmas \ref{lem-001} and \ref{lem-002}, the results follow.
\end{proof}

\begin{proof}[Proof of Theorem \ref{existence-of-TW}]

(1) Thanks to Lemma \ref{lem-003}, for $D\gg D^*$, it follows by comparison principle for parabolic equations that
$$
U(t_2,x;u)<U(t_1,x;u),\quad \forall\ x\in\R, 0\leq t_1<t_2, \forall u\in\tilde{\mathcal{E}}.
$$
Hence the function
$$
U(x;u)=\lim_{t\to\infty}U(t,x;u,c_\kappa), \quad u\in\tilde{\mathcal{E}}
$$
is well defined. Moreover, by estimates for parabolic equations, it follows that
$$
U_{xx}+(c_\kappa-\Psi_x(\cdot;u,c_\kappa))U_x+(a-\chi\Psi_{xx}(\cdot;u,c_\kappa)-b U)U=0,\quad x\in\R,
$$
and $$U(\cdot;u,c_\kappa)\in\tilde{\mathcal{E}} \quad \forall u\in\tilde{\mathcal{E}}. $$
Next we endow $\tilde{\mathcal{E}}$ with the compact open topology. From this point, it follows from the arguments of the proof of \cite[Theorem 4.1]{SaSh3} that the function $$ \tilde{\mathcal{E}}\ni u \mapsto U(\cdot;u,c_\kappa)\in\tilde{\mathcal{E}}$$ is compact and continuous. Hence, by the Schauder's fixed point theorem, it has a fixed point, say $u^*$. Clearly, $(u,v)(t,x)=(u^*,\Psi(\cdot;u^*,c_\kappa))(x-c_\kappa t)$ is a   nontrivial traveling wave solution of \eqref{Main-eq} satisfying \eqref{asyp-beh-at-positive-infinty}. The proof that
$$
\liminf_{x\to-\infty}u^*(x)>0
$$
follows from \cite[Theorem 1.1 (i)]{FhCh}.

If {\bf (H3)} holds, it follows from Lemma \ref{lem-003} (iv) that  for $D\gg D^*$, it holds that
$$ \mathcal{E}\ni u \mapsto U(\cdot;u,c_\kappa)\in \mathcal{E}.$$
Hence $$\liminf_{x\to-\infty}u^*(x)>0.$$ Therefore, by the stability of the positive constant equilibrium established in Theorem \ref{Global-existence and Stability}, it follows that
$$\lim_{x\to-\infty}u^*(x)=\frac{a}{b}.$$
This completes the proof of Theorem \ref{existence-of-TW} (1).

\smallskip

(2) Observe  that $c^*(\tau)=c_{\kappa_\tau^*}$, and,  by \eqref{kappa-star-eq1},
$$
\kappa^*_{\tau}=\min\left\{\sqrt{a},\sqrt{\frac{\lambda +\tau a}{(1-\tau)_+}}\right\}.
$$
This implies that, if $\lambda\ge a$ or $\tau\ge 1$, $\kappa^*_\tau=\sqrt a$ and then
$c^*(\tau)=2\sqrt a$. In the case $\lambda<a$ and $\tau<1$, {\bf (H4)} implies that
$$
\tau\ge \frac{1}{2}\big(1-\frac{\lambda}{a}\big).
$$
This implies that
$$
2\tau a\ge a-\lambda
$$
and then
$$
a\le \frac{\lambda+\tau a}{1-\tau}.
$$
Hence we also have $\kappa_\tau^*=\sqrt a$ and $c^*(\tau)=2\sqrt a$.
(2) then follows from (1).

\smallskip

(3) {Let $\{c_{n}\}_{n\ge 1}$ be a sequence of real numbers satisfying $c_n>c^*(\tau)$ and $c_n\to c^*(\tau)$ as $n\to\infty$. For each $n\geq 1$, let $(U^{c_n,\tau}(x),V^{c_n,\tau}(x))$ denote a traveling wave solution of \eqref{Main-eq} with speed $c_n$ connecting $(0,0)$ and $(\frac{a}{b},\frac{a\mu}{b\lambda})$ given by Theorem \ref{existence-of-TW} (1). For each $n\geq 1$, since the set $\{x\in\R\ :\ U^{c_n,\tau}(x)=\frac{a}{2b}\}$ is bounded and closed, hence compact, then it has a minimal element, say $x_n$. Next, consider the sequence $\{U^n(x),V^n(x)\}_{n\geq 1}$ defined by
$$
(U^n(x),V^n(x))=(U^{c_n,\tau}(x+x_n),V^{c_n,\tau}(x+x_n)), \quad
\forall x\in\R, \ n\geq 1.
$$
Then, for every $n\geq 1$, $(u(t,x),v(t,x))=(U^n(x-c_n t),V^n(x-c_nt))$ it a traveling wave solution of \eqref{Main-eq} with speed $c_n$ satisfying
\begin{equation*}
U^n(-\infty)=\frac{a}{b}, \quad
U^n(\infty)=0,\quad U^n(0)= \frac{a}{2b},\quad \text{and} \quad U^n(x)\geq \frac{a}{2b} \quad \text{for every}\ x\leq 0.
\end{equation*}
Note that
$$\|U^n\|_{\infty}=\|U^{c_n,\tau}\|_{\infty}\leq \frac{a}{b-\chi\mu}, \quad\forall\ n\ge 1. $$
Hence by estimates for parabolic equations, without loss of generality, we may suppose that $(U^n,V^n)\to (U^*,V^*)$ locally uniformly in $C^2(\R)$. Moreover, the function $(U^*,V^*)$ satisfies
\begin{equation}\label{Z-e1}
\begin{cases}
0=U^*_{xx}+(c^*(\tau)-\chi V^*_x)U^*_x+(a-\chi V_{xx}^*-bU^*)U^*,\quad\ x\in\R\cr
0=V^*_{xx}+\tau c^*(\tau)V^*_x-\lambda V^*+\mu U^*,\quad \ x\in\R,
\end{cases}
\end{equation}
\begin{equation}\label{Z-e2}
\|U^*\|_{\infty}\leq\frac{a}{b-\chi\mu},\quad  U^*(0)=\frac{a}{2b}, \quad U^*(x)\geq \frac{a}{2b} \ \forall\ x\leq 0,\quad \text{and} \quad U^*(x)>0,\,\, \forall\ x\in\R.
\end{equation}
Hence, since {\bf (H3)} holds, it follows by the stability of the positive constant equilibrium giving by Theorem \ref{Global-existence and Stability} (2) that $$ \lim_{x\to-\infty}U^*(x)=\frac{a}{b}.$$
So, in order to complete this proof, it remains to show that
\begin{equation}\label{Z-e3}
\limsup_{x\to\infty}U^*(x)=0.
\end{equation}

Suppose by contradiction that \eqref{Z-e3} does not hold. Whence, there is a sequence $\{y_n\}_{n\geq 1}$ with $y_1=0$, $y_n<y_{n+1}$, $y_n \to\infty$ as $n\to \infty$, and
\begin{equation}
\lim_{n\to\infty}U^*(y_n)=\limsup_{x\to\infty}U^*(x)>0.
\end{equation}
Consider a sequence $\{z_n\}_{n\geq 1}$ given by
$$
U^*(z_n)=\min\{U^*(z)\,|\, y_n\leq z\leq y_{n+1}\}, \quad \forall\ n\geq1.
$$
Thus
$$
\lim_{n\to\infty}U^*(z_n)=\inf_{x\in\R}U^*(x).
$$
Note that $\inf_{x\in\R}U^*(x)=0$, otherwise since {\bf (H3)} holds, we would have from Theorem \ref{Global-existence and Stability} (2) that $U^*(x)\equiv \frac{a}{b}$, which contradicts to \eqref{Z-e2}. Thus, there is some $n_0\gg 1$ such that $z_n$ is a local minimum point for every $n\ge n_0$, and hence
\begin{equation}\label{Z-eq4}
U_{xx}^*(z_n)\geq0\quad \text{and}\quad
U_x^*(z_n)=0,\quad \forall n\ge n_0.
\end{equation}

By \eqref{Z-e2}, $\|U^*\|_{\infty}\leq \frac{a}{b-\chi\mu}$, then it follows from the first equation of \eqref{Z-e1}, from \eqref{A-e1} with $\kappa=0$ and $M=\frac{a}{b-\chi\mu}$, that
$$
0\geq U_{xx}^* +(c^*(\tau)-\chi V_x^*)U_x^* +\Big( a-\chi\mu B_{\lambda,c^*(\tau),\tau}\big(\frac{\lambda}{\lambda_1^{c^*(\tau)}}+\frac{\lambda}{\lambda_2^{c^*(\tau)}}\big)\frac{a}{b-\chi\mu}-(b-\chi\mu)U^* \Big)U^*.
$$
Which combined with \eqref{aux-eq2} yields,
\begin{equation}\label{Z-eq5}
0\geq U_{xx}^* +(c^*(\tau)-\chi V_x^*)U_x^* +\left( \frac{a(b-2\chi\mu)}{b-\chi\mu}-(b-\chi\mu)U^* \right)U^*.
\end{equation}
But $\lim_{n\to\infty}U^*(z_n)=0$ and \eqref{Z-eq4} imply that there is $n_1\gg n_0$ such that
$$
U_{xx}^*(z_{n_1})\geq 0, \quad U_x^*(z_{n_1})=0, \quad \frac{a(b-2\chi\mu)}{b-\chi\mu}-U^*(z_{n_1})>0.
$$
This contradicts to \eqref{Z-eq5}, since $U^*(z_{n_1})>0$. Therefore,  \eqref{Z-e3}  holds.
}
\end{proof}

\end{document}